\documentclass[10pt]{amsart}
\usepackage{amsmath, amssymb}
 \usepackage{mathrsfs}
\newcommand{\no}[1]{#1}
\renewcommand{\no}[1]{}
\no{\usepackage{times}\usepackage[subscriptcorrection, slantedGreek, nofontinfo]{mtpro}
\renewcommand{\Delta}{\upDelta}}
\usepackage{color}

\date{\today}
\newtheorem{theorem}{Theorem}[section]
\newtheorem{proposition}{Proposition}[section]
\newtheorem{lemma}{Lemma}[section]

\newtheorem{corollary}{Corollary}[section]

\newtheorem{remark}{Remark}[section]

\numberwithin{equation}{section}

\title[Inverse medium problem]{New Stability  Estimates for the
 Inverse Medium Problem\\ with Internal Data}

\author[Mourad Choulli]{Mourad Choulli\dag}
\address{\dag IECL, UMR CNRS 7502, Universit\'e de Lorraine, Boulevard des Aiguillettes BP 70239 54506 Vandoeuvre Les Nancy cedex- Ile du Saulcy - 57 045 Metz Cedex 01 France}
\email{mourad.choulli@univ-lorraine.fr}

\author[Faouzi Triki]{Faouzi Triki\ddag}
\address{\ddag Laboratoire Jean Kuntzmann,  Universit\'e Joseph Fourier, BP 53, 38041 Grenoble
Cedex 9, France\ddag }
\email{faouzi.triki@imag.fr}

\date{\today}

\begin{document}

\begin{abstract}
A major problem in solving multi-waves inverse problems is the presence of critical points
where the collected data completely vanishes. The set of these
critical points depend on the choice of the boundary conditions, and
can be directly determined from the data itself. To our knowledge, in the most existing
stability results, the boundary conditions are assumed to be close to
a set of CGO solutions where the critical points can be avoided. We establish in the present work new weighted stability estimates for an
electro-acoustic inverse problem without  assumptions
on the presence of critical points. 
 These results show that the Lipschitz  stability 
 far from the critical points  deteriorates  near these points to a  
logarithmic stability.



\end{abstract}

\maketitle



\section{Introduction}
Recently, number of  works~\cite{ABGNS, ABGS, AGNS, ACGRT, BS, BU, SW} have  developped a mathematical framework  for new biomedical imaging modalities based on multi-wave probe of the medium. 
The objective is to stabilize  and improve the resolution of imaging of biological  tissues. 

\smallskip
Different kinds of waves propagate in biological tissues and carry on 
informations on its properties. Each one of them  is more sensitive to
a specific physical parameter  and can be used to provide an accurate 
image of it. For example, ultrasonic waves has influence on
the density, electric waves are sensitive to  the conductivity and optical
waves impact the optical absorption. Imaging modalities based on a single 
wave are known to be ill-posed and suffer from low
resolution~\cite{BT, SU}. Furthermore, the
stability estimates related to these modalities are mostly
logarithmic, i.e. an infinitesimal 
noise in the measured data may be exponentially 
amplified and may give rise to a large error in the 
computed solution~\cite{SU, BT}. One promising way to overcome 
the intrinsic limitation of single wave imaging and provide a stable 
and accurate reconstruction of  physical parameters of  a biological tissue is to
combine different wave-imaging modalities~\cite{BBMT, ACGRT, Ka, Ku}.

\smallskip
A variety of multi-wave imaging approaches are being
introduced and studied over the last decade.
The term multi-wave refers to the fact that two  types of physical
waves are used to probe the medium under study. Usually, the first wave
is sensitive to the contrast
of the desired parameter, 
the other types can carry the information revealed by the first type 
of waves to the boundary of the medium  where measurements can be
taken. There are different types of waves interaction that can be used 
to produce a single imaging system with best contrast and resolution
properties of the two  waves~\cite{FT}: 
the interaction of the first  wave with the tissue can generate a second
kind of wave ; a first wave that carries information about the
contrast of the desired parameter can be locally modulated by a second
wave that has better spatial resolution ; a fast propagating of the first 
wave can be used to acquire a spatio-temporal sequence of the
propagation of a slower transient wave. Typically, the inversion
procedure proceeds in two steps. Informations are retrieved
from the waves of the second type by measurement on the boundary
of the medium. The first inversion
 usually takes the form of a well-posed linear inverse 
source problem and provide internal
data for the waves of the first type that are sensitive to the
contrast of the desired physical parameter of the medium. The
second step consists in recovering the values of this parameter
from the acquired  internal data.

\smallskip
Here, we consider the peculiar case
of electric measurements under elastic perturbations, where 
the medium is probed with acoustic waves while making electric
boundary measurements~\cite{ABCTF, ACGRT, BS, BoT}. The associated internal data
consists in the pointwise value of the electric 
energy density. In this paper, like in~\cite{Tr, BK},
we focus on the second inversion, that is,
to reconstruct the desired physical parameter of the medium from
internal data. We refer the readers to the papers~\cite{ABCTF, ACGRT, BM} for the
modelization of the coupling between the 
acoustic and electric waves  and the way to get such internal
data. 
  
\smallskip
We introduce the mathematical framework of our inverse problem. On a bounded domain $\Omega$ of $\mathbb{R}^n$, $n=2,3$,  with boundary
$\Gamma$, we consider the  second order differential operator 
\[
L_q=\partial _i(a^{ij}\partial _j\, \cdot \,)+q.
\]

Following is a list of assumptions that may be used to derive
the main results in the next section. 

\medskip
\noindent
{\bf a1}. The matrix $(a^{ij}(x))$ is symmetric for any $x\in \Omega$.
 
\smallskip
\noindent
{\bf a2}. (Uniform ellipticity condition) There is a constant $\lambda >0$ so that
\[
a^{ij}(x)\xi _i\xi _j\geq \lambda |\xi |^2,\;\; x\in \Omega ,\; \xi \in \mathbb{R}^n.
\]

\smallskip
\noindent
{\bf a3}. The domain  $\Omega$ is of class $C^{1,1}$.

\smallskip
\noindent
{\bf a4}. $a^{ij}\in  W^{1,\infty} (\Omega )$, $1\leq i,j\leq n$.

\smallskip
\noindent
{\bf a5}. $g\in W^{2-\frac{1}{p}, p}(\Gamma )$ with $p>n$ and it is non identically equal to zero.

\medskip
We further define $\mathscr{Q}_0$ as the subset of $L^\infty(\Omega)$ consisting of functions $q$ such that $0$ is not an eigenvalue of the realization of
the operator $L_q$ on $L^2(\Omega )$ under the Dirichlet boundary
condition.

\smallskip
For $q\in \mathscr{Q}_0$, let $u_q\in H^1(\Omega )$ denotes the unique weak solution of the boundary value problem (abbreviated to BVP in the sequel)
\begin{equation}\label{e1}
\left\{
\begin{array}{llccc}
L_qu=0\;\; &\textrm{in}\;\; \Omega ,
\\
u=g &\textrm{on}\;\; \Gamma .
\end{array}
\right.
\end{equation}

We will see in Section \ref{section3} that in fact $u_q\in W^{2,p}(\Omega )$, for any $q\in  \mathscr{Q}_0$.

\smallskip
We assume in the present work that the matrix  $(a^{ij})$ is
known. We are then concerned with the inverse problem of reconstructing 
the coefficient $q=q(x)$ from the internal over-specified data
\begin{eqnarray}\label{intdata}
I_q(x)=q(x)u_q^2(x), \qquad x\in \Omega.
\end{eqnarray}

In~\cite{Tr},  the second author derived a Lipschitz  
stability estimate for the inverse problem above when 
the  $(a^{ij}(x))$ is the identity  matrix and
when the Dirichlet boundary
 condition  $g$ is  chosen in such a way that $u_q$ does not vanish over
 $\Omega$. It turns out that the
 Lipschitz constant is inversely proportional
to $\min_{\Omega}|u_q|$ and can be very large if this later is close
to zero. 

\smallskip
We aim to establish stability estimates
without such an assumption on the  Dirichlet boundary
 condition  $g$. Obviously, 
 when $u_q$ vanishes at a critical point
 $\hat x$, that is $ u_q(\hat x) = 0$, we expect to lose
informations on $q(x)$ in the surrounding area  of the point
 $\hat x$. Henceforth,  we predict that the stability
estimate in such
areas  will deteriorate.
How worse  it can be ? Can we derive stability estimates
that  reflect our intuition ?

\smallskip
In fact, the presence of critical points where the collected internal 
data completely vanishes is actually a major problem in solving
multi-waves inverse problems. The set of these critical points depend
on the choice of the boundary conditions, and can be directly
determined from the data itself~\cite{HN,Tr, BC}. 
 In most existing stability results, the boundary conditions are
 assumed to be close to a set of CGO solutions (we refer to ~\cite{SU} for more details on such solutions),
 where the citical
 points can be avoided~\cite{ACGRT, ABCTF, BU, BBMT, KS}. The goal of the present work is to avoid using such 
assumptions since they are not realistic in physical point of
view. 

\smallskip
 The rest of this text is structured as follows. The main results
are stated in Section \ref{section2}. The well-posedness and the
regularity of the solution to the direct problem
are provided in Section \ref{section3}. The uniqueness of the inverse
problem is proved in Section \ref{section4}. Weighted stability estimates for the electro-acoustic inverse problems without assumptions on the boundary conditions are given in Section \ref{section5}. Finally, we prove general logarithmic stability estimates  in Section \ref{section6}. 

 \section{Main results}\label{section2}
We state  the  main uniqueness and stability results 
of the paper. 

\begin{theorem}\label{t1}
{\rm (Uniqueness)} We assume that conditions {\bf a1}-{\bf a5} are satisfied. Let $q,\, \widetilde{q}\in \mathscr{Q}_0$ be such that
\[
\frac{\widetilde{q}}{q}\in C(\overline{\Omega})\; \; \mbox{and}\; \; \min_{\overline{\Omega}}\left|\frac{\widetilde{q}}{q}\right|>0,
\]
Then $I_q=I_{\widetilde{q}}$ implies $q=\widetilde{q}$.
\end{theorem}

Next, we define the set of unknown coefficient for which we
will prove a stability estimate. Fix $q_0>0$ and  $0<k<1$. Let 
 $q^\ast \in  \mathscr{Q}_0$ such
 that $0<2q_0\leq q^\ast$. 
 
 \smallskip
 Given $q\in \mathscr{Q}_0$, we let $A_q$ be the unbounded operator $A_qw=-L_qw$ with domain $D(A_q)=\{w\in H_0^1(\Omega );\; L_qw\in L^2(\Omega )\}$. The assumption that $0$ is not an eigenvalue of $A_q$ means that $A_q^{-1}:L^2(\Omega )\rightarrow L^2(\Omega )$ is bounded or, in other words, $0$ belongs to the resolvent set of $A_q$.
 
 \smallskip
 We define $\mathscr{Q}$ to be the set of those functions in $L^\infty (\Omega )$ satisfying
 \begin{equation}\label{eq1}
 \|q-q^\ast\|_{L^\infty (\Omega )}\leq \min 
\left(\frac{k}{\| A_{q^\ast}^{-1}\|_{\mathscr{B}(L^2(\Omega ))}}, q_0\right).
 \end{equation}
 We will see in Section \ref{section3} that $\mathscr{Q}\subset \mathscr{Q}_0$.
 
 \smallskip
  In the sequel $C$, $C_j$, where $j$ is an integer, are generic constants that can only depend  on
 $\Omega$, $(a^{ij})$, $n$,  $g$,   possibly on the a priori bounds we
 make on the unknown function $q$ and, eventually, upon a fixed parameter 
$\theta \in (0,\frac{1}{4})$ 
 
\begin{theorem}\label{t5}{\rm (Weighted stability)}
We assume that assumptions  {\bf a1}-{\bf a5} are fulfilled. Let $q$,
$\widetilde{q}\in \mathscr{Q} \cap W^{1,\infty}(\Omega )$
 satisfying $q-\tilde{q}\in H_0^1(\Omega )$ and
\[
\|q\|_{W^{1,\infty}(\Omega )},\, \|\widetilde q\|_{W^{1,\infty}(\Omega )} \leq M.
\]

Then
\begin{equation}\label{e26}
\left\|\sqrt{I_q} \left(q -\widetilde{q}\right)\right\|_{H^1(\Omega)} \leq 
C \left\| \sqrt{I_q} - \sqrt{I_{\widetilde{q}}} \right\|_{H^1(\Omega)}^{\frac{1}{2}},
\end{equation}

\end{theorem}
The result above shows that a Lipschitz  stability holds far from  
the set of critical points where the solution vanishes.

\smallskip
Next, we  derive   general stability estimates without weight and which hold 
everywhere including in the vicinity of the critical points where 
$u_q$ vanishes. The goal is to get rid of the weight 
 $\sqrt{I_q}$ in the stability~\eqref{e26}. The problem is reduced to solve
the following linear inverse problem:  reconstruct $f(x) \in L^\infty(\Omega)$ 
from the knowledge of $ |u_q(x)| f(x), \, x\in \Omega$. 
Obviously, if $u_q$ does not vanish over $\Omega$, the multiplication operator $M_{|u_q|}$
is invertible, and  as in~\cite{Tr},  a Lipschitz stability estimate  will hold.  Here,  the 
zero set of $u_q$ can  be not empty.  The main idea to overcome this
difficulty consists in quantifying the unique continuation property of the operator $L_q$, that
 is, $u_q$ can not vanish within a not empty open set (see for
 instance Lemmata~\eqref{l3} and~\eqref{le1}).

\smallskip
Prior to state our general stability estimate, we need to make additional assumptions:
 
 \medskip
 \noindent
 {\bf a6}. There exists a positive integer $m$ such that any two disjoints points of $\Omega$ can be connected by a broken line consisting of at most $m$ segments.
 
 \smallskip
\noindent
{\bf a7}. $|g|>0$ on $\Gamma$.

\begin{remark}\label{rm2.1}
{\rm
It is worthwhile to mention that assumption {\bf a6} is equivalent to say $\Omega$ is multiply-starshaped. The notion of multiply-starshaped domain was introduced in \cite{CJ} and it is defined as follows: the open subset $D$ is said mutiply-starshaped if there exists a finite number of points in $D$, say $x_1,\ldots ,x_k$, so that 
\\
(i) $\cup_{i=1}^{k-1}[x_i,x_{i+1}]\subset D$,
\\
(ii) any point in $D$ can be connected by a line segment to at least one of the points $x_i$. 
\\
In this case, any two points in $D$ can be connected by a broken line consisting of at most $k+1$ line segments. Obviously, the case $k=1$ corresponds to the usual definition of a starshaped domain. 
}
\end{remark}

 Henceforth, $\phi $ is a generic function of the form  
 \[
 \phi (s)=C_0\left[|\ln C_1|\ln (s)||^{-1}+s\right],\;\; s>0.
 \]

 \begin{theorem}\label{t2}{\rm (General stability)}
 We assume that conditions {\bf a1}-{\bf a7} hold. We fix $0<\theta
 <\frac{1}{4}$ and $M>0$. For any $q$, $\widetilde{q}\in \mathscr{Q} \cap W^{1,\infty}(\Omega )$
 satisfying $q-\tilde{q}\in H_0^1(\Omega )$ and
\[
\|q\|_{W^{1,\infty}(\Omega )},\, \|\widetilde q\|_{W^{1,\infty}(\Omega )} \leq M.
\] it holds that
 \[
 \|q-\widetilde{q}\|_{L^\infty (\Omega )}\leq  \phi \left(\left\|\sqrt{I_q}-\sqrt{I_{\widetilde{q}}}\right\|_{H^1(\Omega )}^{\theta}\right).
 \]
 \end{theorem}
 
 \smallskip
We still have a weak version of the stability estimate above even if we do not assume that condition {\bf a7} holds. Namely, we prove

 \begin{theorem}\label{t2bis}{\rm (General weak stability)}
Let assumptions {\bf a1}-{\bf a6} hold.  We fix $0<\theta
<\frac{1}{4}$, $M>0$ and $\omega$ an open neighborhood of $\Gamma$ in
$\overline{\Omega}$. For all $q$, $\widetilde{q}\in \mathscr{Q} \cap W^{1,\infty}(\Omega )$
 satisfying $q=\widetilde{q}$ on $\omega\cup \Gamma$  and
\[
\|q\|_{W^{1,\infty}(\Omega )},\, \|\widetilde q\|_{W^{1,\infty}(\Omega )} \leq M.
\]
 it holds that
 \[
 \|q-\widetilde{q}\|_{L^\infty (\Omega )}\leq  \phi \left(\left\|\sqrt{I_q}-\sqrt{I_{\widetilde{q}}}\right\|_{H^1(\Omega )}^{\theta}\right).
 \]
 \end{theorem}
The general stability estimates indicate that the inverse problem of
recovering the coefficient $q$ in the vicinity of the
critical points is severely ill-posed.   The key point in proving our general stability results consists in a estimating the local behaviour of the solution of the BVP \eqref{e1} in terms of $L^2$ norm. To do so, we adapt the method introduced  in \cite{BCJ} to derive a stability estimate for the problem of recovering the surface impedance of an obstacle from the the far field pattern. This method was also used in \cite{CJ} to establish a stability estimate for the problem detecting a corrosion from a boundary measurement.
 
 
 \section{$W^{2,p}$-regularity of the solution of the BVP}\label{section3}
 
 In all of this section, we assume that assumptions {\bf a1}-{\bf a5} are fulfilled.

\begin{theorem}\label{t3}
i) Let $q\in \mathscr{Q}_0$. Then $u_q$, the unique weak solution of the BVP \eqref{e1}, belongs to $W^{2,p}(\Omega )$.
\\
ii) We have the following a priori bound
\begin{equation}\label{3.1}
\| u_q\|_{W^{2,p}(\Omega )}\leq C,\;\; q\in \mathscr{Q}.
\end{equation}
\end{theorem}

\begin{proof}
i) Let $q\in \mathscr{Q}_0$. We pick $G\in W^{2,p}(\Omega )$ such that $G=g$ on $\Gamma$ and we set
$f_q=L_qG$. From our assumptions on $a^{ij}$, $f_q\in L^p (\Omega )$ and
therefore $v_q=u_q-G$ is the weak solution of the following BVP
\[ 
\left\{
\begin{array}{ll}
-L_qv= f_q\;\; &\textrm{in}\;\; \Omega ,
\\
v =0 &\textrm{on}\;\; \Gamma .
\end{array}
\right.
\]
That is $v_q=A_q^{-1}f_q$. In light of the fact that $W^{2,p}(\Omega )$ is continuously embedded in $H^2(\Omega )$, we get from the classical $H^2$-regularity theorem (e.g. \cite{RR}[Theorem 8.53, page 326]), that $u_q=G+A_q^{-1}f_q\in H^2(\Omega )$.

\smallskip
But, $H^2(\Omega )$ is continuously embedded in $L^\infty (\Omega )$ when $n=2$ or $n=3$. Therefore, 
\[
-\partial _i(a^{ij}\partial _ju_q)=qu_q\in L^\infty (\Omega ).
\]
Hence, $u_q \in W^{2,p}(\Omega )$ by \cite{GT}[Theorem 9.15, page 241].

\smallskip
ii) Let $q\in \mathscr{Q}$. If $M_{q^\ast-q}$ is the multiplier by $q^\ast-q$, acting on $L^2(\Omega )$, then $u_q-u_{q^\ast}\in D(A_q)$ and
\begin{equation}\label{3.2}
A_q(u_q-u_{q^\ast})= M_{q^\ast-q}u_{q^\ast}.
\end{equation}

On the other hand, we have
\[
A_q=A_{q^\ast}\left( I+A_{q^\ast}^{-1}M_{q^\ast-q}\right).
\]
Since by our assumption
\[
\|A_{q^\ast}^{-1}\|_{\mathscr{B}(L^2(\Omega ))}\|q-q^\ast\|_{L^\infty (\Omega )}\leq k<1,
\]
the operator $I+A_{q^\ast }^{-1}M_{q^\ast -q }$ is an isomorphism on $L^2(\Omega )$ and
\[
\| \left(I+A_{q^\ast }^{-1}M_{q^\ast -q}\right)^{-1} \|_{\mathscr{B}(L^2(\Omega ))}\leq \frac{1}{1-k}.
\]

Hence, $A_q^{-1}$ is bounded and
\[
\|A_q^{-1}\|_{\mathscr{B}(L^2(\Omega ))}=\|\left(I+A_{q^\ast }^{-1}M_{q^\ast -q}\right)^{-1}A_{q^\ast}^{-1}\|_{\mathscr{B}(L^2(\Omega ))}
\leq \frac{ \|A_{q^\ast}^{-1}\|_{\mathscr{B}(L^2(\Omega ))}}{1-k}.
\]
In light of \eqref{3.2}, we get
\[
\|u_q\|_{L^2(\Omega )}\leq \|u_{q^\ast}\|_{L^2(\Omega )}+\|u_q-u_{q^\ast}\|_{L^2(\Omega )}\leq \|u_{q^\ast}\|_{L^2(\Omega )}\left(1+\frac{k}{k-1}\right).
\]
In combination with \cite{RR}[(8.190), page 326], this estimate implies $\|u_q\|_{H^2(\Omega )}\leq C$ and therefore, bearing in mind that $H^2(\Omega )$ is continuously embedded in $L^\infty (\Omega )$, \[\|u_q\|_{L^\infty(\Omega )}\leq C\]

\smallskip
Finally, from \cite{GT}[Lemma 9.17, page  242], we
find \[\|u_q\|_{W^{2,p}(\Omega )}\leq C_1\|qu_q\|_{L^p (\Omega
  )}+\|L_0 G\|_{L^p(\Omega)} \leq C.\]
\end{proof}

Taking into account that $W^{2,p}(\Omega)$ is continuously embedded into $C^{1,\beta}(\overline{\Omega})$ with $\beta =1-\frac{n}{p}$, we obtain as straightforward consequence of estimate \eqref{3.1} the following corollary:
\begin{corollary}
\begin{equation}\label{3.3}
\| u_q\|_{C^{1,\beta}(\overline{\Omega})}\leq C,\;\; q\in \mathscr{Q}.
\end{equation}
\end{corollary}

\begin{remark}\label{remark3.1}
{\rm
Estimate \eqref{3.3} is crucial in establishing our general uniform stability estimates. Specifically, estimate \eqref{3.3} is necessary to get that $\phi$ in Theorems \ref{t2} and \ref{t2bis} doesn't depend on $q$ and $\widetilde{q}$. Also, the $C^{1,\beta}$-regularity of $u_q$ guaranties that $|u_q|$ is Lipschitz continuous which is a key point in the proof of the weighted stability estimate.
}
\end{remark}


\section{Uniqueness}\label{section4}

In this section, even if it is not necessary, we assume for simplicity that conditions {\bf a1}-{\bf a5} hold true.

\smallskip
The following Caccioppoli's inequality will be useful in the sequel. These kind of inequalities are well known (e.g. \cite{Mo}), but for sake of completeness we give its (short) proof. 

\begin{lemma}\label{l1}
Let $q\in \mathscr{Q}_0$ satisfying $\|q\|_{L^\infty (\Omega)}\leq \Lambda$, for a given $\Lambda >0$. Then there exists a constant $\widehat{C}=\widehat{C}(\Omega ,(a^{ij}),\Lambda )>0$ such that, for any $x\in \Omega$ and $0<r<\frac{1}{2}\mbox{dist}(x,\Gamma )$,
\begin{equation}\label{e4}
\int_{B(x,r)}|\nabla u_q|^2dy\leq \frac{\widehat{C}}{r^2}\int_{B(x,2r)}u_q^2dy.
\end{equation}
\end{lemma}

\begin{proof}
We start by noticing that the following identity holds true in a straightforward way
\begin{equation}\label{e5}
\int_\Omega a^{ij}\partial _i u_q\partial _jv dy =\int_\Omega qu_qv dy,\; \; v\in C_0^1(\Omega ),
\end{equation}
We pick $\chi \in C_c^\infty (B(x,2r))$ satisfying $0\leq \chi \leq 1$, $\chi =1$ in a neighborhood of $B(x,r)$ and $|\partial ^\gamma \chi|\leq Kr^{-|\gamma |}$ for $|\gamma |\leq 2$, where $K$ is a constant not depending on $r$. Then identity \eqref{e5} with $v=\chi u$ gives
\begin{align*}
\int_\Omega \chi a^{ij}\partial _iu_q\partial _ju_q dy&=-\int_\Omega u_q a^{ij}\partial _iu\partial _j \chi dy +\int_\Omega \chi qu_q^2dy
\\
&=-\frac{1}{2}\int_\Omega  a^{ij}\partial _iu_q^2\partial _j \chi dy +\int_\Omega \chi qu_q^2dy
\\
&=\frac{1}{2}\int_\Omega  u_q^2\partial _i\left( a^{ij}\partial _j \chi \right) dy +\int_\Omega \chi qu_q^2dy.
\end{align*}
Therefore, since
\[
\int_\Omega \chi a^{ij}\partial _iu_q\partial _ju_q dy\geq \lambda \int_\Omega \chi |\nabla u_q|^2dy,
\]
 \eqref{e4} follows immediately.
\end{proof}

\begin{theorem}\label{t4} 
 For any  compact subset $K\subset \Omega$, $u^{-2r}_q\in L^1(K)$, for some $r=r(K,u_q)>0$.
\end{theorem}

\begin{proof}
For sake of simplicity, we use in this proof $u$ in place of $u_q$.

\smallskip
We first prove that $u^2$ is locally a Muckenhoupt  weight. We follow the method introduced by Garofalo and Lin in \cite{GL}.

\smallskip
Let $B_{4R}=B(x,4R)\subset \Omega$. According to Proposition 3.1 in \cite{MV} (page 7), we have the following so-called doubling inequality
\begin{equation}\label{e6}
\int_{B_{2r}}u^2dy \leq \widetilde{C}\int_{B_r}u^2dy,\; 0<r\leq R.
\end{equation}
Here and until the end of the proof, $\widetilde{C}$ denotes a generic constant that can depend on $u$ and $R$ but not in $r$.

\smallskip
Since $\partial _j (a^{ij}\partial _iu^2) =2u\partial _j (a^{ij}\partial _iu)+2a^{ij}\partial _iu\partial _j u$, we have 
\[ 
 \partial _j (a^{ij}\partial _iu^2)+2qu^2=2a^{ij}\partial _iu\partial _j u\geq \lambda |\nabla u|^2 \geq 0\; \mbox{in}\; \Omega .
\]
By \cite{GT}[Theorem 9.20, page 244], we have, noting that $u^2\in W^{2,n}(B_{4R})$,
\begin{equation}\label{e7}
\sup_{B_r}u^2\leq \frac{\widetilde{C}}{|B_{2r}|}\int_{B_{2r}} u^2 dy.
\end{equation}
On the other hand, we have trivially
\begin{equation}\label{e8}
\left( \frac{1}{|B_{r}|}\int_{B_{r}} u^{2(1+\delta )} dx \right)^{\frac{1}{1+\delta}}\leq \sup_{B_r}u^2,\; \mbox{for any}\; \delta >0.
\end{equation}

From \eqref{e6}, \eqref{e7} and \eqref{e8}, we obtain
\begin{equation}\label{e9}
\left( \frac{1}{|B_{r}|}\int_{B_{r}} u^{2(1+\delta )} dy \right)^{\frac{1}{1+\delta}}\leq \widetilde{C}\left(\frac{1}{|B_{r}|}\int_{B_{r}} u^2 dy\right),\; \mbox{for any}\; \delta >0.
\end{equation}
In other words, $u^2$ satisfies a reverse H\"older inequality. Therefore, we can apply a theorem by Coifman and Fefferman \cite{CF} (see also \cite{Ko}[Theorem 5.42, page 136]). We get
\begin{equation}\label{e10}
\frac{1}{|B_{r}|}\int_{B_{r}} u^2 dy\left( \frac{1}{|B_{r}|}\int_{B_r}u^{-\frac{2}{\kappa -1}}dy\right)^{\kappa -1}\leq \widetilde{C},
\end{equation}
where $\kappa >1$ is a constant depending on $u$ but not in $r$.

\smallskip
By the usual unique continuation property $u$ cannot vanish identically on $B_R=B_R(x)$. Therefore \eqref{e6} implies
\begin{equation}\label{e11}
\int_{B_R(x)}u^{-2r(x)}dy\leq \widetilde{C}.
\end{equation}
Let $K$ be a compact subset of $\Omega$. Then $K$ can be covered by a finite number, say $N$, of balls $B_i=B_R(x_i)$. Let $(\varphi _i)_{1\leq i\leq N}$ be a partition of unity subordinate to the covering $(B_i)$. If $r=r(K)=\min \{r(x_i);\; 1\leq i\leq N\}$, then
\[
\int_Ku^{-2r}dy=\sum_{i=1}^N\int_Ku^{-2r}\varphi _i dy\leq \sum_{i=1}^N\| \varphi _i\|_\infty \int_{B_i}u^{-2r}dy.
\]
\end{proof}

\begin{proof}[Proof of Theorem \ref{t1}]
As before, for simplicity, we use $u$ (resp. $\widetilde{u}$) in place of $u_q$ (resp. $u_{\widetilde{q}}$).

\smallskip
Let $(\Omega _k)$ be an increasing sequence of open subsets of $\Omega$, $\Omega _k\Subset \Omega$ for each $k$ and $\cup_k\Omega _k=\Omega$. By Theorem ~\ref{t4}, for each $k$,
\[
\frac{u^2}{\widetilde{u}^2}=\frac{\widetilde{q}}{q}\;\; \mbox{a.e.}\; \mbox{in}\; \Omega _k
\]
Therefore,
\[
\frac{u^2}{\widetilde{u}^2}=\frac{\widetilde{q}}{q}\;\; \mbox{a.e.}\; \mbox{in}\;  \Omega .
\]
Changing $\frac{u^2}{\widetilde{u}^2}$ by its continuous representative $\frac{\widetilde{q}}{q}$, we can assume that $\frac{u^2}{\widetilde{u}^2} \in C(\overline{\Omega})$. Moreover, the last identity implies also that $\frac{\widetilde{q}}{q}\geq 0$. Or $\min_{\overline{\Omega}}\left|\frac{\widetilde{q}}{q}\right|>0$. Therefore, $\frac{u}{\widetilde{u}}$ is of constant sign and doesn't vanish. But, $\frac{u}{\widetilde{u}}=1$ on $\{ x\in \Gamma ;\; g(x)\neq 0\}$. Hence,
\begin{equation}\label{e15}
\frac{u}{\widetilde{u}}=\sqrt{\frac{\widetilde{q}}{q}}.
\end{equation}
Let $v=u-\widetilde{u}$. Using $I_q=I_{\widetilde{q}}$ and \eqref{e15}, we obtain by a straightforward computation that $v$ is a solution of the BVP
\[
\left\{
\begin{array}{ll}
-\partial _i(a^{ij}\partial _jv) +\sqrt{\widetilde{q}q}v=0\;\; &\textrm{in}\;\; \Omega ,
\\
v=0 &\textrm{on}\;\; \Gamma .
\end{array}
\right.
\]
Since the operator $-\partial _i(a^{ij}\partial _j\, \cdot \, )
+\sqrt{\widetilde{q}q}$ under Dirichlet boundary condition is strictly
elliptic, we conclude that $v=0$. Hence, $u=\widetilde{u}$  and consequently $q=\widetilde{q}$.
\end{proof}

\section{Weighted stability estimates}\label{section5}

We show that $|u_q|$, $q\in \mathscr{Q}_0$, is a solution of  a certain BVP. To do so, we first prove that $|u_q|$ is Lipschitz continuous. 
  
\begin{proposition} \label{reg|u|}
We have $|u_q|\in C^{0,1}(\overline{\Omega})$, for any $q\in \mathscr{Q}_0$.
\end{proposition}
\begin{proof}
Let $q\in \mathscr{Q}_0$. We recall that, since  $u_q \in H^1(\Omega)$, $|u_q|\in H^1(\Omega)$  and 
\begin{eqnarray} \label{e23}
\partial_i |u_q| = \textrm{sg}_0(u) \partial_i u_q \quad \textrm{a.e. in}\; \Omega .
 \end{eqnarray}
Here 
\[
\textrm{sg}_0(s)= \left\{ \begin{array}{ll} 1\; &s>0,\\ 0\; &s=0\\-1 &s<0\end{array}\right.
\]

We saw in the preceding section that $u\in C^{1,\beta}(\overline{\Omega})$, $\beta =1-\frac{n}{p}$. Hence, $u\in  C^{0,1}(\overline{\Omega})$ because $C^{1,\beta}(\overline{\Omega})$ is continuously embedded in $C^{0,1}(\overline{\Omega})$. Then $|u_q|\in C^{0,1}(\overline{\Omega})$ as a straightforward consequence of the elementary inequality $||s|-|t||\leq |s-t|$, for any scalar numbers $s$ and $t$.
\end{proof}

Let $q\in \mathscr{Q}$. For simplicity, we use in the sequel $u$ instead of $u_q$. It follows from Proposition~\eqref{reg|u|} that
$|u|^2$ lies in $H^1(\Omega)$ and satisfies 
\begin{equation*} 
\partial_i |u|^2 = 2 |u|\partial_i |u| \quad \textrm{a.e. in}\; \Omega. 
\end{equation*}
Obviously, $ \partial_i u^2 = \partial_i |u|^2$, and so 
\begin{equation}\label{dws}
 |u|\partial_i |u|  = u\partial_i u \quad \textrm{a.e. in}\; \Omega. 
\end{equation}

 A straightforward computation gives
\[
\partial_i\left(a^{ij}\partial_j u^2\right)- 2 a^{ij}\partial_iu\partial_j u
=-2qu^2   \quad \textrm{a.e. in}\; \Omega.
\]
Using identity \eqref{e23}, we can rewrite the equation
above as follows
\begin{equation} \label{eee}
\partial_i\left(a^{ij}\partial_j |u|^2\right)- 2 a^{ij}\partial_i|u|\partial_j |u|
=-2qu^2   \quad \textrm{a.e. in}\;\Omega.
\end{equation}
Using the fact that $|u|H_0^1(\Omega )\subset H_0^1(\Omega )$ and integrations by parts, we get
\[
|u|  \partial_i\left(a^{ij}\partial_j |u|\right)= \partial_i\left(a^{ij}|u|\partial_j |u|\right)-a^{ij}\partial_j |u|\partial_i|u|\quad \rm{in}\; H^{-1}(\Omega ).
\]
But the right hand side of this identity belongs to $L^2(\Omega )$. Therefore, 
\begin{equation} \label{dwbis}
 \partial_i\left(a^{ij}|u|\partial_j |u|\right) = 
|u|  \partial_i\left(a^{ij}\partial_j |u|\right)
+ a^{ij}\partial_j |u|\partial_i|u| \quad \textrm{a.e. in}\; \Omega .
\end{equation}
Bearing in mind identities~\eqref{dwbis} and~\eqref{dws}, we obtain by a simple calculation
\begin{equation}\label{e24bis}
\partial_i\left(a^{ij}\partial_j |u|^2\right) = 
2 a^{ij}\partial_i|u|\partial_j |u|
+ 2|u| \partial_i\left(a^{ij}\partial_j |u|\right)   \quad  \textrm{a.e. in}\; \Omega,
\end{equation}
which, combined with~\eqref{eee},  implies 
\begin{equation}\label{e24}
|u|\partial_i\left(a^{ij}\partial_j |u|\right)
=-qu^2   \quad \textrm{a.e. in}\; \Omega.
\end{equation}

Now, as $q \geq q_0>0$ in $\Omega$, $\mu \,=\, \frac{1}{\sqrt q}$ is well defined. 
We set $J=J_q=\sqrt{I_q}$. Then substituting $|u|$ by $J\mu$ 
in \eqref{e24},
we get
\begin{equation}\label{e25}
J\partial_i\left(a^{ij}\partial_j  (J \mu)\right)= -\frac{J^2}{\mu} \quad  \textrm{a.e. in}\; \Omega.
\end{equation}

\begin{proof}[Proof of Theorem~\ref{t5}]
Let $\widetilde{\mu} =\frac{1}{\sqrt{\widetilde{q}}}$ and
$\widetilde{J}=\sqrt{I_{\widetilde{q}}}$. Identity \eqref{e25} with $q$ substituted by $\widetilde{q}$ yields
\begin{equation}\label{e27}
\widetilde J \partial_i\left(a^{ij}\partial_j  (\widetilde{J} \widetilde{\mu})\right) = -\frac{ \widetilde{J}^2}{\widetilde{\mu}} \quad \textrm{a.e. in}\; \Omega.
\end{equation}
Taking the difference of each side of equations \eqref{e25} and \eqref{e27}, we obtain
\begin{align}
J\partial_i\left(a^{ij}\partial_j  (J (\mu-
  \widetilde{\mu}))\right)-\frac{J^2}{\mu\widetilde{\mu}}(\mu-\widetilde{\mu})dx
&= 
\frac{ \widetilde{J}^2-J^2}{\widetilde{\mu}} \label{e27-25}
\\
&+J\partial_i\left(a^{ij}\partial_j  (\widetilde{\mu} 
(\widetilde{J}-J))\right)+ (\widetilde{J}-J) 
\partial_i\left(a^{ij}\partial_j
(\widetilde{\mu}
\widetilde{J})\right)   \quad \textrm{a.e. in}\; \Omega.\nonumber
\end{align}
We multiply  each side of the identity above by $\mu
-\widetilde{\mu}$ and we integrate over $\Omega$. We apply then the
Green's formula to the resulting identity. Taking into account that
$J = \widetilde{J}$ and $\mu = \widetilde \mu $ on $\Gamma$, we obtain
\begin{align*}
\int_{\Omega} a^{ij}\partial_i  (J (\mu- \widetilde{\mu}))\partial_j  (J (\mu- \widetilde{\mu}))dx  + \int_\Omega
\frac{1}{\mu\tilde \mu}J^2(\mu-\tilde \mu)^2dx &=
\int_\Omega \frac{ 1}{\widetilde{\mu}}(J^2-\widetilde{J}^2)(\mu 
- \widetilde{\mu})dx
\\
&\qquad +\int_{\Omega}
a^{ij}\partial_i  (\widetilde{\mu}
(J-\widetilde{J}))\partial_j(J(\mu-\widetilde{\mu}))dx
\\
&\qquad \quad+ \int_{\Omega}a^{ij}\partial_i
(\widetilde{\mu}
\widetilde{J}) \partial_j \left((\widetilde{J}-J) 
(\mu- \widetilde{\mu})\right)dx. 
\end{align*}
We apply the Cauchy-Schwarz's inequality to each term of the right hand side of the above identity. We obtain
\[
\widetilde C_0\|J(\mu-\widetilde{\mu})\|^2_{H^1(\Omega )} \leq
\widetilde C_1 \left(\|J\|_{L^2(\Omega)}+\|\widetilde J\|_{L^2(\Omega)}\right)
\|J-\widetilde{J}\|_{H^1(\Omega )} +\widetilde C_2 
\|J\|_{H^1(\Omega )}
\|J-\widetilde{J}\|_{H^1(\Omega )},
\]
where 
\begin{align*}
&\widetilde C_0 = \min(\lambda, q_0),\\
&\widetilde C_1 = \left(1+ 
\left(\frac{M }{q_0}\right)^{\frac{1}{2}}\right)
 + 
\|a^{ij}\|_{L^\infty(\Omega)} \left(\frac{1}{ \sqrt q_0}+
\frac{M}{q_0}\right),\\
&\widetilde C_2=\|a^{ij}\|_{L^\infty(\Omega)} \left(\frac{1}{ \sqrt q_0}+
\frac{M}{q_0}\right)^2. 
\end{align*}

Clearly, this estimate 
leads to \eqref{e26}.
 \end{proof}

\section{General stability estimates}\label{section6}

We assume in the present section, even if it is not necessary, that
conditions {\bf a1}-{\bf a6} are satisfied.

\subsection{Local behaviour of the solution of the BVP}


We pick $q\in  \mathscr{Q}_0$ and we observe that our assumption on $\Omega$ implies that this later possesses the uniform interior cone property. That is, there exist $R>0$  and $\theta \in ]0,\frac{\pi}{2}[$ such that for all $\widetilde{x}\in \Gamma $, we find $\xi \in \mathbb{S}^{n-1}$ for which
\[
\mathcal{C}(\widetilde{x})=\left\{x\in \mathbb{R}^n;\; |x-\widetilde{x}|<2R,\; (x-\widetilde{x})\cdot \xi >|x-\widetilde{x}|\cos \theta \right\}\subset \Omega .
\]
We set $x_0=\widetilde{x}+R\xi$ and, for $0<r<R\left( \frac{3}{\sin \theta}+1\right)^{-1}$, and we  consider the sequence
\[
x_k=\widetilde{x}+(R-kr)\xi ,\; k\geq 1.
\]
Let $N_0$ be the smallest integer satisfying
\[
(R-(N_0+1)r)\sin \theta \leq 3r.
\]
This and the fact that $(R-N_0r)\sin \theta >3r$ imply
\begin{equation}\label{e16}
\frac{R}{r}-\frac{3}{\sin \theta}-1\leq N_0\leq \frac{R}{r}-\frac{3}{\sin \theta}.
\end{equation}
Let assumption {\bf a7} holds and set $2\eta =\min_{\Gamma}|g|>0$. In light of estimate \eqref{3.3}, we get $|u_q|\geq \eta $ in a neighborhood (independent on $q$) of $\Gamma$ in $\overline{\Omega}$. Or \[
|x_N-\widetilde{x}|=R-N_0r \leq \left( \frac{3}{\sin \theta }+1\right)r.
\]
Therefore, there exists $r^\ast$ such that $|u_q|\geq \eta$ in $B(x_{N_0},r)$, for all $0<r\leq r^\ast$. Consequently,
\begin{equation}\label{e17}
Cr^{\frac{n}{2}}\leq \|u_q\|_{L^2(B(x_{N_0},r)}.
\end{equation}

We shall use the following three spheres Lemma. In the sequel, we assume $\|q\|_{L^\infty (\Omega )}\leq \Lambda$.
\begin{lemma}\label{l2}
There exist $N>0$ and $0<s <1$, that only depend  on $\Omega$, $(a^{ij})$ and 
$\Lambda >0$ such that, for all  $w\in H^1(\Omega )$ satisfying $L_qw=0$ in $\Omega$, $y\in \Omega$ and $0<r< \frac{1}{3 }\mbox{dist}(y,\Gamma )$,  
\[
r\|w\|_{H^1(B(y,2r))}\leq N\|w\|_{H^1(B(y,r))}^s\|w\|_{H^1(B(y,3r))}^{1-s}.
\]
\end{lemma}
For sake of completeness, we prove this lemma in appendix A by means of a Carleman inequality.

\smallskip
 We introduce the following temporary notation
 \[
 I_k=\| u_q\|_{H^1(B(x_k,r))},\;\; 0\leq k\leq N_0.
 \]
 Since $B(x_k,r)\subset B(x_{k-1},2r)$, $1\leq k\leq N_0$, we can apply Lemma \ref{l2}. By an induction argument in $k$, we obtain
 \[
 I_{N_0}\leq \left(\frac{C}{r}\right)^{\frac{1-s^{N_0}}{1-s}}C^{1-s^{N_0}}I_0^{s^{N_0}},
 \]
Shortening  if necessary $r^\ast$, we can always assume that $\frac{C}{r^\ast}\geq 1$. Hence
  \begin{equation}\label{e18}
 I_{N_0}\leq \frac{C}{r^t}I_0^{s^{N_0}},\;\; t=\frac{1}{1-s}.
 \end{equation}
 A combination of \eqref{e17} and \eqref{e18} yields
 \[
 Cr^{\frac{n}{2}+t}\leq I_0^{s^{N_0}}.
 \]
That is 
\begin{equation}\label{e19}
Cr^{\frac{n}{2}+t}\leq \|u\|_{H^1(B(x_0,r))}^{s^{N_0}}.
 \end{equation}
 
 Now, let $y_0\in \Omega$ such that the segment $[x_0,y_0]$ lies entirely in $\Omega$. Set
 \[
 d=|x_0-y_0|\;\; \mbox{and}\;\; \zeta =\frac{y_0-x_0}{|x_0-y_0|}.
 \]
 Consider the sequence, where $0<2r<d$,
 \[
 y_k=y_0-k(2r)\zeta ,\;\; k\geq 1.
 \]
 We have
 \[
 |y_k-x_0|=d-k(2r).
 \]
Let $N_1$ be the smallest integer such that $d-N_1(2r)\leq r$, or equivalently
 \[
\frac{d}{2r}-\frac{1}{2}\leq N_1 <\frac{d}{2r}+\frac{1}{2}.
 \]
 
 Using again Lemma \ref{l2}, we obtain
 \begin{equation}\label{e20}
 \widetilde{C}r^t\|u\|_{H^1(B(y_{N_1},2r))}\leq \|u\|_{H^1(B(y_0,2r))}^{s^{N_1}}.
 \end{equation}
 
 Since $|y_{N_1}-x_0|=d-N_1(2r)\leq r$, $B(x_0,r)\subset B(y_{N_1}, 2r)$. From \eqref{e19} and \eqref{e20}, shortening again $r^\ast$ if necessary, we have
 \[
\widetilde{C}r^t \left(Cr^{\frac{n}{2}+t}\right)^{\frac{1}{s^{N_0}}}\leq \|u\|_{H^1(B(y_0,2r))}^{s^{N_1}}.
 \]
Reducing  again  $r^\ast$ if necessary, we can assume that $\widetilde{C}r^t<1$, $0<r\leq r^\ast$. Therefore, we find
 \[
 \left(Cr\right)^{\frac{\frac{n}{2}+2t}{s^{N_0+N_1}}}\leq \|u\|_{H^1(B(y_0,2r))}.
 \]
 
 Let us observe that we can repeat the argument between $x_0$ and
 $y_0$ to  $y_0$ and $z_0\in \Omega$ such that  $[y_0,z_0]\subset
 \Omega$  and so on. By this way and since we have assumed that each
 two points of $\Omega$ can be connected by a broken line consisting
 of at most $m$ segments,  we have the following result.

 \begin{lemma}\label{l3}
Let assumption {\bf a7} be satisfied. There are constants $c>0$ and $r^\ast>0$, that can only 
depend on data and $\Lambda$, such that for any $x\in \overline{\Omega}$,
 \[
 e^{-ce^{\frac{c}{r}}}\leq \|u_q\|_{H^1(B(x,r)\cap \Omega )},\; \; 0<r\leq r^\ast .
 \]
 \end{lemma}
 
\begin{remark}\label{remark6.1}
We observe that we implicitly use the following property for
establishing Lemma \ref{l3}, which is a direct consequence of the
regularity of $\Omega$:  if $\widetilde{x}\in \Gamma$, then
$|B(\widetilde{x},r)\cap \Omega | \geq \kappa |B(\widetilde{x},r)|$,
$0<r$, where the constant
 $\kappa \in (0,1)$ does not depend on $\widetilde{x}$ and $r$.
\end{remark}
 
 Observing again that $|u_q|\geq \eta$ in a neighborhood of $\Gamma$
 in $\overline{\Omega}$, with $\eta>0$ is as above, a combination of Lemma
 ~\ref{l3} and Lemma ~\ref{l1} (Caccippoli's inequality) leads to
 
 \begin{corollary}\label{c1}
 Under assumption {\bf a7}, there are constants $c>0$ and $r^\ast>0$, that can depend only on data and $\Lambda$, such that for any $x\in \overline{\Omega}$,
 \[
 e^{-ce^{\frac{c}{r}}}\leq \|u_q\|_{L^2(B(x,r)\cap \Omega )},\; \; 0<r\leq r^\ast .
 \]
 \end{corollary}
 
 From this corollary, we get the key lemma that we will use to establish general stability estimates for our inverse problem.
 
 \begin{lemma}\label{l4}
 Assume that {\bf a7} is fulfilled. There exists $\delta ^\ast$ so that, for any $x\in \Omega$ and $0<\delta \leq \delta ^\ast$,
 \[
 \{y\in B(x,r^\ast )\cap \Omega ;\; u_q(x)^2\geq \delta \}\neq \emptyset .
 \]
 \end{lemma}
 \begin{proof} 
If the result does not hold we can find a sequence $(\delta _k)$, $0<\delta _k\leq \frac{1}{k}$ such that
 \[
 \{y\in B(x,r^\ast)\cap \Omega ;\; u_q(x)^2\geq \delta _k\}=\emptyset .
 \]
Then
 \[
 u_q^2\leq \frac{1}{k}\; \; \mbox{in}\; B(x,r^\ast )\cap \Omega .
 \]
 In light of Corollary~\ref{c1}, we obtain
 \[
  e^{-ce^{\frac{c}{r^\ast}}}\leq \frac{1}{\sqrt k}|B(x,r^\ast )\cap
  \Omega |\leq \frac{|\Omega |}{\sqrt k},\; \mbox{for all}\; k\geq 1.
  \]
  This is impossible. Hence, the desired contradiction follows.
  \end{proof}

Further, we recall that the semi-norm $[f]_\alpha$, for $f\in C^\alpha(\overline
\Omega)$, is given as follows:
\[
[f]_\alpha =\sup \{|f(x)-f(y)||x-y|^{-\alpha};\; x,y\in \overline{\Omega} ,\; x\neq y\}.
\]

 \begin{proposition}\label{p2}
 Let $M>0$ be given and assume that assumption {\bf a7} holds. Then, for any $q\in \mathscr{Q}$, $f\in C^\alpha (\overline{\Omega})$ satisfying $\|f\|_{C^\alpha (\overline{\Omega})}\leq M$, we have
  \[
  \|f\|_{L^\infty (\Omega )}\leq \phi \left(\|fu_q^2\|_{L^\infty (\Omega)}\right).
  \]
  \end{proposition}
   
  \begin{proof}
 Let $\delta ^\ast$ be as in Corollary ~\ref{c1}, $f\in C^\alpha (\overline{\Omega})$ and $0<\delta \leq \delta ^\ast$. For $x\in \overline{\Omega}$, we consider two cases: (a) $u_q(x)^2\geq \delta$  and (b) $u_q(x)^2\leq \delta$.
 
 \smallskip
(a) $u_q(x)^2\geq \delta$. We have
\begin{equation}\label{e21}
|f(x)|\leq \frac{1}{\delta}|f(x)u_q(x)^2|.
\end{equation}

(b) $u(x)^2\leq \delta$. We set
\[
r=\sup \{0<\rho ;\; u_q ^2<\delta \; \mbox{on}\; B(x,\rho )\cap \overline{\Omega}\}.
\]
By Lemma~\ref{l4}
\[
\{x\in B(x, r^\ast )\cap \overline{\Omega};\; u_q(x)^2\geq \delta\}\neq\emptyset .
\]
Hence, $r \leq r^\ast$ and 
\[
\partial B(x,r)\cap \{x\in B(x, r^\ast )\cap\overline{\Omega};\; u_q(x)^2 \geq \delta\}\neq\emptyset .
\]

Let $y\in\partial B(x,r)$ be such that  $u_q(y)^2\geq \delta $. We have,
\[
|f(x)|\leq |f(x)-f(y)|+|f(y)|\leq [f]_\alpha|x-y|^\alpha +\frac{1}{\delta}|f(y)u_q(y)^2|
\]
and then
\[
|f(x)|\leq Mr^\alpha +\frac{1}{\delta}|f(y)u_q(y)^2|.
\]

This and \eqref{e21} show
\begin{equation}\label{e22}
\|f\|_{L^\infty (\Omega )}\leq Mr^\alpha +\frac{1}{\delta}\|fu_q^2\|_{L^\infty (\Omega )}.
\end{equation}
Since $u_q^2\leq \delta$ in $B(x,r )\cap \overline{\Omega}$, Corollary~\ref{c1} implies
\[
 e^{-ce^{\frac{c}{r}}}\leq \sqrt \delta |B(x,r )\cap \Omega |\leq
 \sqrt \delta |\Omega| ,
\]
or equivalently
\[
r\leq \frac{c}{\ln \left( \frac{1}{c}|\ln |\Omega |\sqrt \delta |\right)} .
\]
This estimate in \eqref{e22} yields
\[
\|f\|_{L^\infty (\Omega )}\leq \frac{C_0}{\left[\ln \left( C_1|\ln
      |\Omega |\sqrt \delta |\right)\right]^\alpha} +\frac{1}{\delta}\|fu_q^2\|_{L^\infty (\Omega )}.
\]
Setting $e^s=C_1|\ln |\Omega |\sqrt \delta |$, the last inequality can be rewritten as
\[
\|f\|_{L^\infty (\Omega )}\leq \frac{C_0}{s^\alpha} +|\Omega |e^{\frac{2}{C_1}e^s}\|fu_q^2\|_{L^\infty (\Omega )}.
\]
An usual minimization argument with respect to $s$ leads to the existence of $\epsilon >0$, depending only on data, so that
\begin{equation}\label{equ1}
\|f\|_{L^\infty (\Omega )}\leq \psi \left(\|fu_q^2\|_{L^\infty (\Omega )} \right),
\end{equation}
with $\psi$ a function of the form
\[
\psi (s)=C_2|\ln C_3|\ln (s)||^{-1/\alpha}
\]
provided that $\|fu_q^2\|_{L^\infty (\Omega )} \leq \epsilon$.

\smallskip
When $\|fu_q^2\|_{L^\infty (\Omega )} \geq \epsilon$, we have trivially
\begin{equation}\label{equ2}
\|f\|_{L^\infty (\Omega )}\leq \frac{M}{\epsilon}\|fu_q^2\|_{L^\infty (\Omega )}.
\end{equation}

The desired estimate follows by combining \eqref{equ1} and \eqref{equ2}.
\end{proof}

We now turn our attention to the case without assumption {\bf a7}. First, we observe that we still have a weaker version of Corollary \ref{c1}. In fact all our analysis leading to Lemma \ref{l3} holds whenever we start with $\widetilde{x}\in \Gamma$ satisfying $2\eta =|g(\widetilde{x})|\neq 0$ and take $x_0$ in a neighborhood of $\widetilde{x}$ such that $|u|\geq \eta$ in a ball around $x_0$ (we note that, as we have seen before, such a ball can be chosen independently on $q$).

 \begin{lemma}\label{le1}
Let $K$ be a compact subset of $\Omega$. There exist $r^\ast >0$ and $c>0$, depending only on data, $\Lambda$ and $K$, so that, for any $x\in K$,
 \[
 e^{-ce^{\frac{c}{r}}}\leq \|u\|_{H^1(B(x,r))},\; \; 0<r\leq r^\ast .
 \]
 \end{lemma}

In light of this lemma, an adaptation of the proofs of Lemma \ref{l4} and Proposition \ref{p2} yields
  \begin{proposition}\label{pr1}
Given $M>0$ and $K$ a compact subset of $\Omega$, for any $f\in C^\alpha (\overline{\Omega})$ satisfying $\|f\|_{C^\alpha (\overline{\Omega})}\leq M$, we have
  \[
  \|f\|_{L^\infty (K)}\leq \phi \left(\|fu^2\|_{L^\infty (\Omega)}\right).
  \]
  \end{proposition}


 \subsection{Proof of stability estimates}
 
 \begin{proof}[Proof of Theorem \ref{t2}]
 Set $U=u_q^2$ and $\widetilde{U}=u_{\widetilde{q}}^2$. Then a straightforward computation shows 
 \[
 \sqrt{U}-\sqrt{\widetilde{U}}=\sqrt{I}\left( \frac{1}{\sqrt{q}}-\frac{1}{\sqrt{\widetilde{q}}}\right) +\frac{1}{\sqrt{\widetilde{q}}}\left( \sqrt{I}-\sqrt{\widetilde{I}}\right).
 \]
 This identity, in combination with Theorem \ref{t5}, gives 
 \begin{eqnarray}
 \left\|U-\widetilde{U}\right\|_{H^1(\Omega )} &\leq& \left(
   \left\|\sqrt U\right\|_{H^1(\Omega )} + 
\left\|\sqrt{\widetilde{U}}\right\|_{H^1(\Omega )} \right)
\left\| \sqrt{U}-\sqrt{\widetilde{U}}\right\|_{H^1(\Omega )}\nonumber \\
&\leq& C \left(
   \left\|\sqrt U\right\|_{H^1(\Omega )} + 
\left\|\sqrt{\widetilde{U}}\right\|_{H^1(\Omega )} \right)
\left(1+ \|  \sqrt{U}\|_{H^1(\Omega )}\right)
 \left\| \sqrt{I}-\sqrt{\widetilde{I}}\right\|_{H^1(\Omega )}^{\frac{1}{2}}.\label{e28}
 \end{eqnarray}
 
 On the other hand, since $L^\infty (\Omega )$ is continuously 
embedded in $H^{2-2\theta}$,
 
 \begin{align*}
 \left\| U-\widetilde{U}\right\|_{L^\infty (\Omega )}
\leq C \left\| U-\widetilde{U}\right\|_{H^{2-2\theta}(\Omega )}.
 \end{align*}
 This, the interpolation inequality (e.g. \cite{LM}[Remark 9.1, page 49])
 \[
 \left\|w\right\|_{H^{2-2\theta}(\Omega )}\leq C \left\|w\right\|_{H^1(\Omega )}^{2\theta} \left\|w\right\|_{H^2(\Omega )}^{1-2\theta},\; w\in H^2(\Omega ),
 \]
 and estimate \eqref{3.1} imply
 \begin{equation}\label{e29}
  \left\| U-\widetilde{U}\right\|_{L^\infty (\Omega )} \leq C \left\| \sqrt{I}-\sqrt{\widetilde{I}}\right\|_{H^1(\Omega )}^{\theta}.
 \end{equation}
But
 \[
 \left(q-\widetilde{q}\right)u^2= I-\widetilde{I}+\widetilde{q}\left(U-\widetilde{U}\right).
 \]
 Then \eqref{e29} yields
 \[
  \left\|  \left(q-\widetilde{q}\right)u^2\right\|_{L^\infty (\Omega )} \leq C \left\| \sqrt{I}-\sqrt{\widetilde{I}}\right\|_{H^1(\Omega )}^{\theta}.
  \]
  We complete the proof by applying Proposition \ref{p2} with $f=q-\widetilde{q}$.
 \end{proof}
 
  \begin{proof}[Proof of Theorem \ref{t2bis}]
  Similar to the previous one. We have only to apply Proposition \ref{pr1} instead of Proposition \ref{p2}.
 \end{proof}
\section{Conclusion}
This paper has investigated the inverse problem of recovering
a coefficient of a Helmholtz operator from internal data without
assumptions on the presence of critical points. It is shown 
through different stability estimates that
the reconstruction of the coefficient is accurate in  areas
 far from the critical points and  deteriorates near these points.
The optimality of the stability estimates will be investigated 
in future works.

\section{Acknowledgements}
The research of the second author is supported by the project Inverse
Problems and Applications funded by LABEX Persyval-Lab 2013-2015.
 
 \appendix
\section{The three spheres inequality}\label{appendix}

As we said before, the proof of the three spheres inequality for the
$H^1$ norm we present here is based on a Carleman inequality. Since the proof for $L_q$ is not so different from the one of a general second order operator in divergence form, we give the proof for this later. 
 
 \smallskip
 We first consider the following second order operator in divergence form:
\[
L=\mbox{div}(A\nabla \, \cdot ),
\]
where $A=(a^{ij})$ is a matrix with  coefficients in $W^{1,\infty}(\Omega )$ and there exists $\kappa >0$ such that
\begin{equation}\label{a9}
A(x)\xi \cdot \xi \geq \kappa |\xi |^2\;\; \mbox{for any}\; x\in \Omega \; \mbox{and}\; \xi \in \mathbb{R}^n.
\end{equation}

Let $\psi \in C^2(\overline{\omega})$ having no critical point in $\overline{\Omega}$, we set $\varphi =e^{\lambda \psi}$.

\begin{theorem}\label{thm1} (Carleman inequality)
There exist three positive constants $C$, $\lambda _0$ and $\tau _0$, that can depend only on $\psi$, $\Omega$ and a bound of $W^{1,\infty}$ norms of $a^{ij}$, $1\leq i,j\leq n$, such that the following inequality holds true:
\begin{equation}\label{a8}
C\int_\Omega \left (\lambda ^4\tau ^3\varphi ^3v^2+\lambda ^2\tau \varphi |\nabla v|^2 \right)e^{2\tau \varphi} dx \leq \int_\Omega (Lw)^2e^{2\tau \varphi}dx+\int_\Gamma \left( \lambda^3\tau ^3\varphi ^3v^2+\lambda \tau \varphi |\nabla v|^2\right)e^{2\tau \varphi} d\sigma ,
\end{equation}
for all $v\in H^2(\Omega )$, $\lambda \geq \lambda _0$ and $\tau \geq \tau _0$.
\end{theorem}

\begin{proof}
We let $\Phi =e^{-\tau \varphi}$ and $w\in H^2(\Omega)$. Then straightforward computations give
\[
Pw=[\Phi ^{-1}L \Phi ]w=P_1w+P_2w+cw.
\]
where,
\begin{align*}
P_1w&=aw+\mbox{div}\, ( A\nabla w),
\\
P_2w&= B\cdot \nabla w+bw,
\end{align*}
with
\begin{align*}
a&=a(x,\lambda ,\tau )= \lambda ^2\tau ^2 \varphi ^2 |\nabla \psi |_A^2,
\\
B&=B(x,\lambda ,\tau )=-2\lambda \tau \varphi A\nabla \psi ,
\\
b&=b(x,\lambda ,\tau )=-2\lambda ^2\tau \varphi|\nabla \psi |_A^2,
\\
c&=c(x,\lambda ,\tau )=-\lambda \tau \varphi \mbox{div}\, ( A\nabla \psi  )+\lambda ^2\tau \varphi |\nabla \psi |_A^2.
\end{align*}
Here
\[
|\nabla \psi |_A=\sqrt{A\nabla \psi \cdot \nabla \psi}.
\]
We have
\begin{equation}\label{a1}
\int_\Omega awB\cdot \nabla wdx=\frac{1}{2}\int_\Omega aB\cdot \nabla w^2dx=-\frac{1}{2}\int_\Omega \mbox{div}(aB) w^2dx+\frac{1}{2}\int_\Gamma aB\cdot \nu w^2d\sigma
\end{equation}
and
\begin{align}\label{a2}
\qquad \int_\Omega \mbox{div}\, ( A\nabla w)B\cdot \nabla wdx&=-\int_\Omega A\nabla w\cdot \nabla (B\cdot \nabla w)dx+\int_\Gamma B\cdot \nabla wA\nabla w\cdot\nu d\sigma
\\
&=-\int_\Omega B'\nabla w\cdot A\nabla wdx-\int_\Omega\nabla ^2wB\cdot A\nabla wdx+\int_\Gamma B\cdot \nabla w\nabla w\cdot\nu d\sigma.\nonumber
\end{align}
Here, $B'=(\partial _iB_j)$ is the jacobian matrix of $B$ and $\nabla ^2w=(\partial ^2_{ij}w)$ is the hessian matrix of $w$.

\smallskip
But,
\[
\int_\Omega B_i\partial ^2_{ij}w a^{ik}\partial _kwdx=-\int_\Omega B_i a^{jk}\partial ^2_{ik}w \partial ^2_{ij}w\partial _j wdx-\int_\Omega\partial _iB_ia^{jk}\partial _kw\partial _j wdx+\int_\Gamma B_i\nu _i a^{jk}\partial _kw\partial _jwd\sigma.
\]
Therefore,
\begin{equation}\label{a3}
\int_\Omega\nabla ^2wB\cdot A\nabla wdx=-\frac{1}{2}\int_\Omega [\mbox{div}(B)A+\widetilde{A}\nabla w]\cdot \nabla wdx+\frac{1}{2}\int_\Gamma |\nabla w|_A^2B\cdot \nu  d\sigma ,
\end{equation}
with $\widetilde{A}=(\widetilde{a}^{ij})$,
\[
\widetilde{a}^{ij}=B\cdot \nabla a^{ij}.
\]

It follows from \eqref{a2} and \eqref{a3},
\begin{align}\label{a4}
\int_\Omega \mbox{div}\, ( A\nabla w)B\cdot \nabla wdx=\frac{1}{2}\int_\Omega  \Big(-2B' &+\mbox{div}(B)A+\widetilde{A} \Big)\nabla w\cdot\nabla wdx 
\\
&+\int_\Gamma B\cdot \nabla w\nabla w\cdot\nu d\sigma - \frac{1}{2}\int_\Gamma |\nabla w|_A^2B\cdot \nu  d\sigma . \nonumber
\end{align}

A new integration by parts yields
\[
\int_\Omega \mbox{div}\, ( A\nabla w) bwdx=-\int_\Omega b|\nabla w|_A^2dx-\int_\Omega w\nabla b\cdot A\nabla wdx+\int_\Gamma bwA\nabla w\cdot \nu d\sigma .
\]
This and the following inequality
\[
-\int_\Omega w\nabla b\cdot A\nabla wdx\geq -\int_\Omega (\lambda ^2\varphi )^{-1}|\nabla b|_A^2w^2dx-\int_\Omega \lambda ^2\varphi |\nabla w|_A^2dx ,
\]
imply
\begin{equation}\label{a5}
\int_\Omega \mbox{div}\, ( A\nabla w) bwdx\geq -\int_\Omega (b+\lambda ^2\varphi )|\nabla w|_A^2dx-\int_\Omega (\lambda ^2\varphi )^{-1}|\nabla b|_A^2w^2dx+\int_\Gamma bwA\nabla w\cdot \nu d\sigma .
\end{equation}

Now a combination of \eqref{a1}, \eqref{a4} and \eqref{a5} leads
\begin{equation}\label{a6}
\int_\Omega P_1wP_2wdx -\int_\Omega c^2w^2dx\geq \int_\Omega fw^2dx+\int_\Omega F\nabla w\cdot \nabla w dx+\int_\Gamma g(w)d\sigma ,
\end{equation}
where,
\begin{align*}
f&=-\frac{1}{2}\mbox{div}(aB)+ab-(\lambda ^2\varphi )^{-1}|\nabla b|_A^2-c^2,
\\
F&=-B'+\frac{1}{2}\Big(\mbox{div}(B)A +\widetilde{A}\Big) -(b+\lambda ^2\varphi )A,
\\
g(w)&=\frac{1}{2}aw^2B\cdot \nu-\frac{1}{2}|\nabla w|_A^2B\cdot \nu+B\cdot \nabla wA\nabla w \cdot \nu+bwA\nabla w \cdot \nu .
\end{align*}

Using the elementary inequality $(s-t)^2\geq s^2/2-t^2$, $s$, $t>0$, we obtain
\[
\|Pw\|_2^2\geq (\|P_1w+P_2w\|_2-\|cw\|_2)^2\geq \frac{1}{2}\|P_1w+P_2w\|_2^2-\|cw\|_2^2\geq \int_\Omega P_1wP_2w dx-\int_\Omega c^2w^2dx.
\]
In light of \eqref{a6}, we get
\begin{equation}\label{a7}
\|Lw\|_2^2\geq \int_\Omega fw^2dx+\int_\Omega F\nabla w\cdot \nabla w dx+\int_\Gamma g(w)d\sigma .
\end{equation}

After some straightforward computations, we find that there exist three positive constants $C_0$, $C_1$, $\lambda _0$ and $\tau_0$, that can depend only on $\psi$, $\Omega$ and a bound of $W^{1,\infty}$ norms of $a^{ij}$, $1\leq i,j\leq n$, such that for all $\lambda \geq \lambda _0$ and $\tau \geq \tau_0$,
\begin{align*}
&f\geq C_0 \lambda ^4\tau ^3\varphi ^3,
\\
&F\xi \cdot \xi \geq C_0\lambda ^2\tau \varphi |\xi |^2,\;\; \textrm{for any}\; \xi \in \mathbb{R}^n,
\\
&|g(w)|\leq C_1\left( \lambda ^3\tau ^3\varphi ^3w^2+\lambda \tau \varphi |\nabla w|^2 \right).
\end{align*}
Hence,
\[
C\int_\Omega \left (\lambda ^4\tau ^3\varphi ^3w^2+\lambda ^2\tau \varphi |\nabla w|^2 \right) dx \leq \int_\Omega (Pw)^2dx+\int_\Gamma \left( \lambda^3\tau ^3\varphi ^3w^2+\lambda \tau \varphi |\nabla w|^2\right) d\sigma .
\]
As usual, we take $w=\Phi ^{-1}v$, $v\in H^1(\Omega )$, in the previous inequality to derive the Carleman estimate \eqref{a8}.
\end{proof}

\begin{remark}\label{rem1}
i) We notice that \eqref{a8} is still valid when $v\in H^1(\Omega )$ with $Lv\in L^2(\Omega )$. Also, the inequality \eqref{a8} can be extended to an operator $\widetilde{L}$ of the form
\[
\widetilde{L}= L+L',
\]
where $L'$ is a first order operator with bounded coefficients. For the present case, the constants $C$, $\lambda _0$ and $\tau _0$ in the statement of Theorem \ref{thm1} may depend also on a bound of $L^\infty$ norms of the coefficients of $L'$.

\smallskip
ii) We can substitute $L$ by an operator $L_r$ of the same form, whose coefficients $a^{ij}$ depend on the parameter $r$, $r$ belonging to some set $I$. Let $L' _r$ be a first order operator with coefficients depending also on the parameter $r$. Under the assumption that the coefficients of $L_r$ are uniformly bounded in $W^{1,\infty}(\Omega )$ with respect to $r$, the coefficients of $L_r'$ are uniformly bounded in $L^\infty (\Omega )$ with respect to the parameter $r$ and the ellipticity condition \eqref{a9} holds with $\kappa$ independent on $r$, we can paraphrase the proof of Theorem \ref{thm1} and i). We find that \eqref{a8} is true when $\widetilde{L}$ is substituted by $\widetilde{L}= L_r+L_r'$, with constants $C$, $\lambda _0$ and $\tau _0$, independent on $r$. That is we have the following result: there exist three constants $C$, $\lambda _0$ and $\tau _0$, that can depend on the uniform bound of the coefficients of $L_r$ in $W^{1,\infty}(\Omega )$ and the uniform bound of the coefficients of $L_r'$ in $L^\infty (\Omega )$, such that for any $v\in H^1(\Omega )$ satisfying $\widetilde{L}_rv\in L^2(\Omega )$, 
$\lambda \geq \lambda _0$ and $\tau \geq \tau _0$, 
\begin{equation}\label{a10}
C\int_\Omega \left (\lambda ^4\tau ^3\varphi ^3v^2+\lambda ^2\tau \varphi |\nabla v|^2 \right)e^{2\tau \varphi} dx \leq \int_\Omega (\widetilde{L}_rw)^2e^{2\tau \varphi}dx +\int_\Gamma \left( \lambda^3\tau ^3\varphi ^3v^2+\lambda \tau \varphi |\nabla v|^2\right)e^{2\tau \varphi} d\sigma .
\end{equation}
\end{remark}

\begin{lemma}(Three spheres inequality)\label{lem2}
There exist $C>0$ and $0<s <1$, that only depend on a bound of $W^{1,\infty}$ norms of the coefficients of $L$ and a bound of $L^\infty$ norms of the coefficients of $L'$,  such that, for all  $v\in H^1(\Omega )$ satisfying $\widetilde{L}v=0$ in $\Omega$, $y\in \Omega$ and $0<r< \frac{1}{3 }\mbox{dist}(y,\Gamma )$,  
\[
r\|v\|_{H^1(B(y,2r))}\leq C\|v\|_{H^1(B(y,r))}^s\|v\|_{H^1(B(y,3r))}^{1-s}.
\]
\end{lemma}
\begin{proof}
Let $v\in H^1(\Omega )$ satisfying $\widetilde{L}v=0$, set $B(i)=B(0,i)$, $i=1,2,3$ and $r_0=\frac{1}{3}\mbox{diam}(\Omega )$. Fix $y\in \Omega$ and  $0<r< \frac{1}{3}\mbox{dist}(y,\Gamma ) (\leq r_0)$.  Let
\[
w(x)=v(rx+y),\; x\in B(3),
\]
Clearly, $\widetilde{L}_r w=0$, with an operator $\widetilde{L}_r$ as in ii) of Remark \ref{rem1}. 

\smallskip
Let $\chi \in C_c^\infty (U)$ satisfying $0\leq \chi \leq 1$ and  $\chi =1$ in $K$, with
\[
U=\{x\in \mathbb{R}^n;\; 1/2<|x|<3\},\quad K=\{x\in \mathbb{R}^n;\; 1\leq |x|\leq 5/2\}.
\]
We get by applying \eqref{a10} to $\chi w$, with $U$ in place of $\Omega$, $\lambda \geq \lambda _0$ and $\tau \geq \tau _0$,
\begin{equation}\label{a11}
C\int_{B(2)\setminus B(1)} \left (\lambda ^4\tau ^3\varphi ^3w^2+\lambda ^2\tau \varphi |\nabla w|^2 \right)e^{2\tau \varphi} dx \leq \int_{B(3)} (\widetilde{L}_r(\chi w))^2e^{2\tau \varphi}dx
\end{equation}
From $\widetilde{L}_r w=0$ and the properties of $\chi$, we obtain in a straightforward manner that 
\[
\mbox{supp}(\widetilde{L}_r(\chi w))\subset \{1/2 \leq |x|\leq 1\}\cup \{5/2\leq |x|\leq 3\}
\]
and
\[
(\widetilde{L}_r(\chi w))^2 \leq \Lambda (w^2+|\nabla w|^2 ),
\]
where $\Lambda =\Lambda (r_0)$ is independent on $r$. Therefore, fixing $\lambda$ and shortening $\tau _0$ if necessary, \eqref{a11} implies, for $\tau \geq \tau _0$,
\begin{equation}\label{a12}
C\int_{B(2)} \left (w^2+|\nabla w|^2 \right)e^{2\tau \varphi} dx\leq \int_{B(1)} \left (w^2+|\nabla w|^2 \right)e^{2\tau \varphi} dx+\int_{\{5/2\leq |x|\leq 3\}} \left (w^2+|\nabla w|^2 \right)e^{2\tau \varphi} dx.
\end{equation}

Let us now specify $\varphi$. The choice of $\varphi (x)=-|x|^2$ (which is with no critical point in $U$) in \eqref{a12} gives, for $\tau \geq \tau _0$,
\begin{equation}\label{a13}
C\int_{B(2)} \left (w^2+|\nabla w|^2 \right) dx \leq e^{\alpha \tau}\int_{B(1)} \left(w^2+|\nabla w|^2 \right) dx+e^{-\beta \tau}\int_{B(3)} \left (w^2+|\nabla w|^2 \right)dx,
\end{equation}
where
\begin{align*}
&\alpha =\left(1-e^{-2\lambda}\right)
\\
&\beta =2\left(e^{-2\lambda}-e^{-\frac{5}{2}\lambda}\right) .
\end{align*}
 
 We introduce the following temporary notations
 \begin{align*}
 &P=\int_{B(1)} \left(w^2+|\nabla w|^2 \right) dx,
 \\
 &Q=C\int_{B(2)} \left (w^2+|\nabla w|^2 \right) dx,
  \\
 &R=\int_{B(3)} \left (w^2+|\nabla w|^2 \right)dx.
 \end{align*}
Then \eqref{a13} reads
\begin{equation}\label{a14}
Q\leq e^{\alpha \tau}P +e^{-\beta \tau}R,\;\; \tau \geq \tau _0.
\end{equation}
Let 
\[
\tau _1=\frac{\ln (R/P)}{\alpha +\beta}.
\]
If $\tau _1\geq \tau _0$, then $\tau =\tau _1$ in \eqref{a14} yields
\begin{equation}\label{a15}
Q\leq P^{\frac{\alpha}{\alpha +\beta}}R^{\frac{\beta}{\alpha +\beta}}.
\end{equation}
If $\tau _1<\tau _0$, $R<e^{(\alpha +\beta )\tau _0}P$ and then
\begin{equation}\label{a16}
Q\leq R=R^{\frac{\alpha}{\alpha +\beta}}R^{\frac{\beta}{\alpha +\beta}}\leq e^{\alpha \tau _0}P^{\frac{\alpha}{\alpha +\beta}}R^{\frac{\beta}{\alpha +\beta}}.
\end{equation}

Summing up, we get that one inequalities \eqref{a15} and \eqref{a16} holds. That is, we have in terms of our original notations
\[
\|w\|_{H^1(B(2))}\leq C\|w\|_{H^1(B(1))}^s\|w\|_{H^1(B(3))}^{1-s}.
\]
We complete the proof by noting that 
\begin{equation}\label{a17}
c_0 r^{1-n/2}\|v\|_{H^1(B(y,ir))}\leq \|v\|_{H^1(B(y,ir))}\leq c_1 r^{-n/2}\|v\|_{H^1(B(y,ir))},\;\; i=1,2,3,
\end{equation}
where
\[
c_0 =\min (1,r_0), \quad c_1=\max (1,r_0).
\]
Finally, we observe that estimate \eqref{a17} can be obtained in a simple way after making a change of variable.
\end{proof}


\end{document}